\documentclass[11pt,leqno]{article}

\usepackage[active]{srcltx}

\usepackage{amsfonts,latexsym,amsmath,amssymb,amsthm}
\usepackage{fullpage}
\newtheorem{theorem}{Theorem}[section]
\newtheorem{lemma}[theorem]{Lemma}

\newtheorem{application}[theorem]{Application}

\theoremstyle{definition}

\theoremstyle{remark}

\newtheorem*{note*}{Note}

\numberwithin{equation}{section}

\makeatletter

\newcommand{\rank}{\mathop{\operator@font rank}}
\newcommand{\conv}{\mathop{\operator@font conv}}
\newcommand{\vol}{\mathop{\operator@font vol}}
\newcommand{\onetagright}{\tagsleft@false}
\makeatother

\newcommand{\ls}{\leqslant}
\newcommand{\gr}{\geqslant}
\renewcommand{\epsilon}{\varepsilon}

\usepackage{color}

\usepackage{hyperref}

\begin{document}
\small

\title{\bf Reverse Brascamp--Lieb inequality and the dual Bollob\'{a}s--Thomason inequality}

\author{Dimitris-Marios Liakopoulos}

\date{}

\maketitle

\begin{abstract}
\footnotesize We prove that if $f:{\mathbb R}^n\to [0,\infty )$ is an integrable log-concave function with $f(0)=1$
and $F_1,\ldots ,F_r$ are subspaces of ${\mathbb R}^n$ such that $sI_n=\sum_{i=1}^rc_iP_i$ where $I_n$ is the identity
operator and $P_i$ is the orthogonal projection onto $F_i$ then
$$n^n\int_{{\mathbb R}^n}f(y)^ndy\gr\prod_{i=1}^r\left (\int_{F_i}f(x_i)dx_i\right )^{c_i/s}.$$
As an application we obtain the dual version of the Bollob\'{a}s--Thomason inequality: if $K$ is a convex body in ${\mathbb R}^n$ with
$0\in {\rm int}(K)$ and $(\sigma_1,\ldots ,\sigma_r)$ is an $s$-uniform cover of $[n]$ then
$$|K|^s\gr\frac{1}{(n!)^s}\prod_{i=1}^r|\sigma_i|!\prod_{i=1}^r|K\cap F_i|.$$
This is a sharp generalization of Meyer's dual Loomis--Whitney inequality.
\end{abstract}

\section{Introduction}

The purpose of this article is to establish the dual version of the uniform cover inequality of Bollob\'{a}s and
Thomason. We fix an orthonormal basis $\{e_1,\ldots ,e_n\}$ of ${\mathbb R}^n$ and recall that the not necessarily
distinct non-empty sets $\sigma_1,\ldots ,\sigma_r\subseteq [n]:=\{1,\ldots ,n\}$
form an $s$-uniform cover of $[n]$ for some $s\gr 1$ if every $j\in [n]$ belongs to exactly $s$ of the sets $\sigma_i$.
The main result of \cite{Bollobas-Thomason-1995} estimates the volume of a compact set in terms of the volumes of its
coordinate projections that correspond to a uniform cover of $[n]$.

\begin{theorem}[Bollob\'{a}s-Thomason]\label{th:BT}Let $r\gr 1$ and $(\sigma_1,\ldots ,\sigma_r)$ be an
$s$-uniform cover of $[n]$. For every compact subset $K$ of ${\mathbb R}^n$, which is the closure
of its interior, we have
\begin{equation}\label{eq:intro-BT-1}|K|^s\ls\prod_{i=1}^r|P_{F_{\sigma_i}}(K)|,\end{equation}
where $F_{\tau }={\rm span}\{e_j:j\in\tau \}$ and $P_F$ denotes the orthogonal projection of ${\mathbb R}^n$
onto $F$.
\end{theorem}

Throughout this article, for any non-empty compact set in ${\mathbb R}^n$ we write $|A|$ for the volume of $A$ in the
affine subspace ${\rm aff}(A)$. A special case of Theorem \ref{th:BT} is the Loomis--Whitney inequality \cite{Loomis-Whitney-1949}; one has
\begin{equation}\label{eq:intro-BT-2}|K|^{n-1}\ls\prod_{i=1}^n|P_i(K)|\end{equation}
where $P_i:=P_{e_i^{\perp }}$, and equality holds if and only if $K$ is a coordinate box,
i.e. a rectangular parallelepiped whose sides are parallel to the coordinate axes. This follows from
the observation that the sets $\sigma_i=[n]\setminus\{i\}$ form an $(n-1)$-uniform cover of $[n]$.

Meyer proved in \cite{Meyer-1988} an inequality which is dual to the Loomis--Whitney inequality. If $K$ is a convex body
in ${\mathbb R}^n$ then
\begin{equation}\label{eq:intro-BT-3}|K|^{n-1}\gr \frac{n!}{n^n}\prod_{i=1}^n|K\cap e_i^{\perp }|,\end{equation}
where $K\cap F$ denotes the section of $K$ with a subspace $F$. Equality holds in \eqref{eq:intro-BT-3}
if and only if $K={\rm conv}(\{\pm \lambda_1e_1,\ldots ,\pm \lambda_ne_n\})$ for some $\lambda_i>0$.
We prove the following dual Bollob\'{a}s--Thomason inequality.

\begin{theorem}\label{th:dual-BT}Let $K$ be a convex body in ${\mathbb R}^n$ with $0\in {\rm int}(K)$ and $(\sigma_1,\ldots ,\sigma_r)$ be an $s$-uniform cover
of $[n]$. Then,
\begin{equation}\label{eq:intro-BT-4}|K|^s\gr\frac{1}{(n!)^s}\prod_{i=1}^r|\sigma_i|!\prod_{i=1}^r|K\cap F_{\sigma_i}|.\end{equation}
\end{theorem}

It is  not hard to check that \eqref{eq:intro-BT-4} is sharp; it becomes equality for any $s$-uniform cover of $[n]$ if $K$ is the
cross-polytope $B_1^n={\rm conv}(\{\pm e_1,\ldots ,\pm e_n\})$.

An essentially equivalent way to state Theorem \ref{th:BT} (see \cite{Bollobas-Thomason-1995}) is the fact that for every
compact subset $K$ of ${\mathbb R}^n$, which is the closure of its interior, we can find a coordinate box such that $|B|=|K|$ and
\begin{equation}\label{eq:BT-equivalent}|P_{F_{\sigma }}(B)|\ls |P_{F_{\sigma }}(K)|\end{equation}
for every $\sigma\subseteq [n]$. Theorem \ref{th:dual-BT} has a similar equivalent formulation.

\begin{theorem}\label{th:dual-BT-equivalent}Let $K$ be a convex body in ${\mathbb R}^n$ with $0\in {\rm int}(K)$.
There exists an affine cross-polytope $C={\rm conv}(\{\pm \lambda_1e_1,\ldots ,\pm \lambda_ne_n\})$, where $\lambda_i>0$, such that
$|C|=|K|$ and $|C\cap F_{\sigma }|\gr |K\cap F_{\sigma }|$ for every $\sigma\subseteq [n]$.
\end{theorem}

Theorem \ref{th:dual-BT}, and its equivalent version Theorem \ref{th:dual-BT-equivalent}, is a consequence of a functional inequality
which is proved in Section~\ref{sec:dual-BT}. We denote by ${\cal F}({\mathbb R}^n)$
the class of log-concave integrable functions $f:{\mathbb R}^n\to [0,\infty )$.

\begin{theorem}\label{th:dual-functional}Let $f\in {\cal F}({\mathbb R}^n)$ with $f(0)=1$ and $(\sigma_1,\ldots ,\sigma_r)$ be an $s$-uniform cover
of $[n]$. Then,
$$n^n\int_{{\mathbb R}^n}f(y)^ndy\gr\prod_{i=1}^r\left (\int_{F_i}f(x_i)dx_i\right )^{1/s}.$$
\end{theorem}

Moreover, we obtain more general inequalities which imply several of the known extensions of the Loomis--Whitney and Meyer inequalities; see
Section~\ref{sec:BT} and Section~\ref{sec:dual-BT} for the statements and details. Our main tool is Barthe's multidimensional generalization
of Ball's geometric Brascamp-Lieb inequality (see \cite{Ball-handbook}) and its reverse form; see \cite[Theorem~6]{Barthe-1998}. The connection with the problems that
we discuss in Section \ref{sec:BT} was communicated by F.~Barthe to A.~Giannopoulos after a talk in MSRI and the author is grateful to them for the information which
has been the starting point for this work.

Let us also mention that the Bollob\'{a}s--Thomason inequality plays a key role in the recent work \cite{BGL} of S.~Brazitikos,
A.~Giannopoulos and the author that provides local versions of the Loomis--Whitney inequality for coordinate projections
of convex bodies; see also \cite{Alonso-Artstein} for further results in this direction. It is conceivable that one might exploit
the dual inequality of Theorem \ref{th:dual-BT} to obtain analogous local inequalities for sections. Isomorphic inequalities
of this type appear in \cite{BGL} where they are proved by different methods.

In Section \ref{sec:BT} we describe the way one can derive both the Loomis--Whitney and the Bollob\'{a}s--Thomason inequality, as well
as other extensions of them, as consequences of the multidimensional geometric Brascamp--Lieb inequality. The main new results of this
work are presented in Section \ref{sec:dual-BT}; the main tool is Barthe's inequality. We refer to the books \cite{Schneider-book} and
\cite{AGA-book} for standard notation and facts from convex geometric analysis.

\section{Brascamp-Lieb inequality and uniform cover inequalities}\label{sec:BT}

In what follows we say that the subspaces $F_1,\ldots ,F_r$ form an $s$-uniform cover of ${\mathbb R}^n$ with weights $c_1,\ldots ,c_r>0$ for some $s>0$ if
\begin{equation}\label{eq:uniform-weights}sI_n=\sum_{i=1}^rc_iP_i,\end{equation}
where $I_n$ is the identity operator and $P_i$ is the orthogonal projection of ${\mathbb R}^n$ onto $F_i$. We prove the next general result.

\begin{theorem}\label{th:general-uc-1}Let $F_1,\ldots ,F_r$ be subspaces that form an $s$-uniform cover of ${\mathbb R}^n$
with weights $c_1,\ldots ,c_r>0$. For every compact subset $K$ of ${\mathbb R}^n$ we have
\begin{equation}|K|^s\ls\prod_{i=1}^r|P_{F_i}(K)|^{c_i}.\end{equation}
\end{theorem}

The proof is an almost direct consequence of Barthe's multidimensional geometric Brascamp-Lieb inequality. The statement is given below; the
reverse inequality \eqref{eq:barthe-2} will be our main tool in the next section.

\begin{theorem}[Barthe]\label{th:barthe}Let $r,n\in {\mathbb N}$. For $i=1,\ldots r$, let $F_i$ be a $d_i$-dimensional subspace
of ${\mathbb R}^n$ and $P_i$ be the orthogonal projection onto $F_i$. If
\begin{equation*}I_n=\sum_{i=1}^rc_iP_i\end{equation*}
for some $c_1,\ldots ,c_r>0$ then for all non-negative integrable functions $f_i:F_i\to {\mathbb R}$ we have
\begin{equation}\label{eq:barthe-1}\int_{{\mathbb R}^n}\prod_{i=1}^rf_i^{c_i}(P_ix)\,dx\ls \prod_{i=1}^r\left (\int_{F_i}f_i\right )^{c_i}\end{equation}
and
\begin{equation}\label{eq:barthe-2}\int_{{\mathbb R}^n}^{\ast }\sup\left\{ \prod_{i=1}^rf_i^{c_i}(x_i):x=\sum_{i=1}^rc_ix_i, x_i\in F_i\right\}\,dx
\gr \prod_{i=1}^r\left (\int_{F_i}f_i\right )^{c_i}.\end{equation}
\end{theorem}

In the statement above, $\int^{\ast }$ stands for the outer integral and in the right hand side
the integral on $F_i$ is with respect to the Lebesgue measure on $F_i$ which is compatible to the
given Euclidean structure.

\smallskip

\begin{proof}[Proof of Theorem \ref{th:general-uc-1}] Let $d_i={\rm dim}(F_i)$ and note that
\begin{equation*}ns={\rm tr}(sI_n)=\sum_{i=1}^rc_i\cdot {\rm tr}(P_i)=c_1d_1+\cdots +c_rd_r.\end{equation*}
Given a compact subset $K$ of ${\mathbb R}^n$ we define $f_i:F_i\to [0,\infty )$ by
$f_i={\mathbf 1}_{P_i(K)}$. Note that if $x\in K$ then $f_i(P_ix)=1$ for all $i=1,\ldots ,r$. Therefore,
\begin{equation*}{\mathbf 1}_K(x)\ls \prod_{i=1}^rf_i^{\frac{c_i}{s}}(P_ix)\end{equation*}
for all $x\in {\mathbb R}^n$. From Theorem \ref{th:barthe} we get
\begin{align*}
|K| =\int_{{\mathbb R}^n}{\mathbf 1}_K(x)\,dx \ls \int_{{\mathbb R}^n}\prod_{i=1}^rf_i^{\frac{c_i}{s}}(P_ix)\,dx
\ls \prod_{i=1}^r\left (\int_{F_i}f_i\right )^{\frac{c_i}{s}}=\prod_{i=1}^r|P_i(K)|^{\frac{c_i}{s}},
\end{align*}
which shows that $|K|^s\ls\prod_{i=1}^r|P_i(K)|^{c_i}$ as claimed.\end{proof}

\begin{application}[Bollob\'{a}s-Thomason]\label{appl:Bollobas-Thomason}\rm It is not hard to see that the Bollob\'{a}s-Thomason inequality may be proved in the same way. Note that if $(\sigma_1,\ldots ,\sigma_r)$ is
an $s$-uniform cover of $[n]$ then the projections $P_i:=P_{F_{\sigma_i}}$ satisfy
\begin{equation*}sI_n=\sum_{i=1}^rP_i.\end{equation*}
Therefore, for any compact subset $K$ of ${\mathbb R}^n$ we may apply Theorem \ref{th:general-uc-1} with $c_1=\cdots =c_r=1$
to get
\begin{equation*}|K|^s \ls\prod_{i=1}^r|P_i(K)|.
\end{equation*}
which is exactly the statement of Theorem \ref{th:BT}.

As a special case of Theorem \ref{th:general-uc-1} we also obtain the following inequality of Bollob\'{a}s and Thomason \cite{Bollobas-Thomason-1995}.
Let $\mathcal{C}$ be a finite collection of subsets of $[n]$, which is not necessarily a uniform cover. Suppose that to each $\sigma\in\mathcal{C}$ we associate
a positive real weight $w(\sigma )$ in such a way that, for each $i\in [n]$, $\sum\{w(\sigma ):\, i\in \sigma\in\mathcal{C}\}=1.$ Then,
it is clear that
\begin{equation*}I_n=\sum_{\sigma\in \mathcal{C}}w(\sigma )P_{F_{\sigma }},\end{equation*}
and Theorem \ref{th:general-uc-1} shows that
\begin{equation*}|K|\ls\prod_{\sigma\in\mathcal{C}}|P_{F_{\sigma }}(K)|^{w(\sigma )}.\end{equation*}
\end{application}

\begin{application}[Ball's inequality]\label{appl:Ball}\rm Let $u_1,\ldots ,u_m$ be unit vectors in ${\mathbb R}^n$ and
$c_1,\ldots ,c_m$ be positive real numbers such that John's condition
\begin{equation*}I_n=\sum_{i=1}^mc_iu_i\otimes u_i\end{equation*}
is satisfied. Using the one-dimensional geometric Brascamp--Lieb inequality,
Ball proved in \cite{Ball-1991} that for every centered convex body $K$ in ${\mathbb R}^n$,
\begin{equation}\label{eq:LW-ball}|K|^{n-1}\ls\prod_{i=1}^m|P_{u_i^{\perp }}(K)|^{c_i}.\end{equation}
The equality cases are the same with the ones in the Loomis--Whitney inequality.
Let us briefly explain how Theorem \ref{th:general-uc-1} implies \eqref{eq:LW-ball}. We observe that if $P_i=P_{u_i^{\perp }}$
then $u_i\otimes u_i=I_n-P_i$, and hence John's condition may be written as $I_n=\sum_{i=1}^mc_i(I_n-P_i)$, which implies that
\begin{equation}\label{eq:LW-ball-1}(n-1)I_n=\sum_{i=1}^mc_iP_i,\end{equation}
if we take into account the fact that $\sum_{i=1}^mc_i=n$. Then, given a (more generally) compact subset $K$ of ${\mathbb R}^n$
we may apply Theorem \ref{th:general-uc-1} with $s=n-1$ to get
\begin{equation}\label{eq:LW-ball-2}|K|^{n-1}\ls\prod_{i=1}^m|P_{u_i^{\perp }}(K)|^{c_i}.\end{equation}
\end{application}

\section{Dual Bollob\'{a}s-Thomason inequality}\label{sec:dual-BT}

We start with a proof of a more general version of Theorem \ref{th:dual-functional}. Recall that ${\cal F}({\mathbb R}^n)$ denotes
the class of log-concave integrable functions $f:{\mathbb R}^n\to [0,\infty )$.

\begin{theorem}\label{th:dual-1}Let $f\in {\cal F}({\mathbb R}^n)$ with $f(0)=1$ and $F_1,\ldots ,F_r$ be subspaces of ${\mathbb R}^n$
that form an $s$-uniform cover of ${\mathbb R}^n$ with weights $c_1,\ldots ,c_r>0$. Then,
\begin{equation*}n^n\int_{{\mathbb R}^n}f(y)^ndy\gr\prod_{i=1}^r\left (\int_{F_i}f(x_i)dx_i\right )^{c_i/s}.\end{equation*}
\end{theorem}

\begin{proof}Our assumption $I_n=\sum_{i=1}^r\frac{c_i}{s}P_{F_i}$ implies that
\begin{equation*}ns={\rm tr}(sI_n)=\sum_{i=1}^rc_i\cdot {\rm tr}(P_{F_i})=\sum_{i=1}^rc_id_i,\end{equation*}
where $d_i={\rm dim}(F_i)$. Let $z\in {\mathbb R}^n$ and $x_i\in F_i$, $i\in [r]$ such that $z=\sum_{i=1}^r\frac{c_i}{s}x_i$.
Then,
\begin{equation*}\frac{z}{n}=\sum_{i=1}^r\frac{c_id_i}{sn}\cdot\frac{x_i}{d_i},\end{equation*}
and since $f\in {\cal F}({\mathbb R}^n)$ and $\sum_{i=1}^r\frac{c_id_i}{sn}=1$ we have
\begin{equation*}f(z/n)\gr \prod_{i=1}^rf(x_i/d_i)^{\frac{c_id_i}{ns}}.\end{equation*}
Since $f(0)=1$, for every $i\in [r]$ we see that $f(x_i/d_i)\gr f(x_i)^{1/d_i}f(0)^{1-1/d_i}=f(x_i)^{1/d_i}$. It follows
that
\begin{equation*}f(z/n)\gr \prod_{i=1}^rf(x_i)^{\frac{1}{d_i}\cdot\frac{c_id_i}{ns}}=\prod_{i=1}^rf(x_i)^{\frac{c_i}{ns}},\end{equation*}
and hence
\begin{equation*}f(z/n)^n\gr\prod_{i=1}^rf(x_i)^{c_i/s}.\end{equation*}
This shows that
\begin{equation*}f(z/n)^n\gr\sup\left\{\prod_{i=1}^rf(x_i)^{c_i/s}:z=\sum_{i=1}^r\frac{c_i}{s}x_i,x_i\in F_i\right\}.\end{equation*}
Then, by the multidimensional reverse Brascamp-Lieb inequality \eqref{eq:barthe-2} we have that
\begin{align*}\int_{{\mathbb R}^n}f(z/n)^ndz &\gr \int^{\ast }_{{\mathbb R}^n}\sup\left\{\prod_{i=1}^rf(x_i)^{c_i/s}:
z=\sum_{i=1}^r\frac{c_i}{s}x_i,x_i\in F_i\right\}\,dz \\
&\gr\prod_{i=1}^r\left (\int_{F_i}f(x_i)dx_i\right )^{c_i/s}.\end{align*}
Making the change of variables $y=z/n$ we conclude the proof. \end{proof}

\smallskip

Out main geometric application of Theorem \ref{th:dual-1} is the next general uniform cover inequality for sections of a convex body.

\begin{theorem}\label{th:dual-2}Let $K$ be a convex body in ${\mathbb R}^n$ with $0\in {\rm int}(K)$ and $F_1,\ldots ,F_r$ be subspaces of ${\mathbb R}^n$,
with ${\rm dim}(F_i)=d_i$, that form an $s$-uniform cover of ${\mathbb R}^n$ with weights $c_1,\ldots ,c_r>0$. Then,
$$|K|^s\gr\frac{1}{(n!)^s}\prod_{i=1}^r(d_i!)^{c_i}\prod_{i=1}^r|K\cap F_i|^{c_i}.$$
\end{theorem}

\begin{proof}We apply Theorem \ref{th:dual-1} for the function $f(y)=e^{-\|y\|_K}$, where $\|y\|_K:=\min\{t>0:y\in tK\}$ is the Minkowski
functional of $K$. Note that $f\in {\cal F}({\mathbb R}^n)$ and $f(0)=1$. We have
\begin{align*}n^n\int_{{\mathbb R}^n}f(y)^ndy &= n^n\int_{{\mathbb R}^n}e^{-n\|y\|_K}dy =n^n\int_{{\mathbb R}^n}e^{-\|y\|_{\frac{1}{n}K}}dy\\
&= n^n\,n!\left | \frac{1}{n}K\right |=n!\,|K|,
\end{align*}
and for every $i\in [r]$ we have
\begin{equation*}\int_{F_i}f(x_i)dx_i=\int_{F_i}e^{-\|x_i\|_K}dx_i=\int_{F_i}e^{-\|x_i\|_{K\cap F_i}}dx_i=d_i!\,|K\cap F_i|.\end{equation*}
It follows that
\begin{equation*}n!|K|\gr\prod_{i=1}^r\big (d_i!\,|K\cap F_i|\big )^{c_i/s}=\prod_{i=1}^r(d_i!)^{c_i/s}\prod_{i=1}^r|K\cap F_i|^{c_i/s},\end{equation*}
and the theorem follows.\end{proof}

\begin{application}[dual Bollob\'{a}s--Thomason]\label{appl:dual-BT-1}\rm Theorem \ref{th:dual-2} has several
straightforward applications. First, let $(\sigma_1,\ldots ,\sigma_r)$ be an $s$-uniform cover
of $[n]$. Setting $F_i=F_{\sigma_i}={\rm span}(\{e_j:j\in\sigma_i\})$, $i\in [r]$, we have $sI_n=\sum_{i=1}^rP_{F_i}$.
Thus, we obtain the dual Bollob\'{a}s-Thomason inequality of Theorem \ref{th:dual-BT}:
If $K$ is a convex body in ${\mathbb R}^n$ with $0\in {\rm int}(K)$ and $(\sigma_1,\ldots ,\sigma_r)$ is an $s$-uniform cover
of $[n]$ then
\begin{equation*}|K|^s\gr\frac{1}{(n!)^s}\prod_{i=1}^r|\sigma_i|!\prod_{i=1}^r|K\cap F_i|.\end{equation*}
In the particular case $F_i=e_i^{\perp }$, $i\in [n]$ we have $(n-1)I_n=\sum_{i=1}^nP_{e_i^{\perp }}$, and applying Theorem \ref{th:dual-BT}
with $s=n-1$ and $|\sigma_i|={\rm dim}(F_i)=n-1$ we recover Meyer's inequality
\begin{equation*}|K|^{n-1}\gr\frac{n!}{n^n}\prod_{i=1}^n|K\cap e_i^{\perp }|\end{equation*}
for any convex body $K$ in ${\mathbb R}^n$ with $0\in {\rm int}(K)$, because \begin{equation*}\frac{1}{(n!)^{n-1}}\prod_{i=1}^n|\sigma_i|!=\frac{1}{(n!)^{n-1}}\prod_{i=1}^n(n-1)!=\frac{[(n-1)!]^n}{(n!)^{n-1}}=
\frac{(n-1)!}{n^{n-1}}=\frac{n!}{n^n}.\end{equation*}
\end{application}

Theorem \ref{th:dual-BT-equivalent} can be obtained from Theorem \ref{th:dual-BT} by an argument which is basically the same with the
one used by Bollob\'{a}s and Thomason for the proof of \eqref{eq:BT-equivalent}. In what follows, we say that a uniform cover of
$[n]$ is irreducible if it cannot be written as a disjoint union of two other uniform covers of $[n]$. In \cite{Bollobas-Thomason-1995}
it is shown that the number of irreducible uniform covers of $[n]$ is finite.

\medskip

\begin{proof}[Proof of Theorem~\ref{th:dual-BT-equivalent}] Let $K$ be a convex body in ${\mathbb R}^n$ with
$0\in {\rm int}(K)$. Theorem \ref{th:dual-BT} states that for every integer $s\gr 1$ and any non-trivial irreducible
$s$-uniform cover $(\sigma_1,\ldots ,\sigma_r)$ of $[n]$ we have that $(n!|K|)^s\gr \prod_{j=1}^r\big (|\sigma_j|!\,|K\cap F_{\sigma_j}|\big )$.
Moreover, applying Theorem \ref{th:dual-BT} for the $1$-uniform cover $(\{i\},i\in\tau )$ of $\tau\subseteq [n]$ we see that
$|\tau |!|K\cap F_{\tau }|\gr\prod_{i\in\tau }|K\cap F_{\{i\}}|$. Since there are finitely many irreducible uniform covers of $[n]$,
we have a finite number of inequalities as above, satisfied by the elements of the set $\{ |\sigma |!|K\cap F_{\sigma }|:\sigma\subseteq [n]\}$.

Let $\{t_{\sigma }:\sigma\subseteq [n]\}$ be a set of positive reals with $t_{\sigma }\gr |\sigma |!|K\cap F_{\sigma }|$ and
$t_{[n]}=n!|K|$, which are maximal with respect to satisfying all the above inequalities if we replace $|\sigma |!|K\cap F_{\sigma }|$
by $t_{\sigma }$ for all $\sigma\subseteq [n]$. Then, we know that $\prod_{j=1}^rt_{\sigma_j}\ls (n!|K|)^s$ for every (not necessarily
irreducible) $s$-uniform cover $(\sigma_1,\ldots ,\sigma_r)$ of $[n]$.

Since $t_{\{i\}}$, $i\in [n]$, are maximal, we see that for every $i\in [n]$ we can find an inequality involving $t_{\{i\}}$ which
is equality. If this inequality is of the first kind then there exists an $s_i$-uniform cover $\overline{\sigma }(i)=(\sigma_1,\ldots ,\sigma_r)$ of $[n]$ with
$\sigma_j=\{i\}$ for some $j$, such that $(n!|K|)^{s_i}=\prod_{j=1}^rt_{\sigma_j}$. The same is true if the inequality is of the second kind,
i.e. if we have an equality of the type $\prod_{l\in\tau }t_{\{l\}}=t_{\tau }$ for some $\tau\subseteq [n]$ with $i\in\tau $.
Because, by the maximality of $t_{\tau }$ we can find an $s_i$-uniform cover $(\sigma_1,\ldots ,\sigma_r)$ of $[n]$ such that $\tau=\sigma_{j_0}$
for some $j_0$, and then $\overline{\sigma }(i):=(\sigma_j, j\neq j_0)\cup (\{i\}:i\in\tau )$ is again an $s_i$-uniform cover of $[n]$.

Now, we define $\overline{\sigma }=\bigcup_{i=1}^n\overline{\sigma }(i)$ and $s=\sum_{i=1}^ns_i$. Then, $\overline{\sigma }$ is an $s$-uniform cover
of $[n]$, we have $\{i\}\in\overline{\sigma }$ for all $i=1,\ldots ,n$ and
\begin{equation}\label{eq:dual-BT-equivalent-1}\prod_{\sigma\in\overline{\sigma }}t_{\sigma }=(n!|K|)^s.\end{equation}
Since $\overline{\sigma }^{\prime }:=\overline{\sigma }\setminus (\{i\}:i\in [n])$ is an $(s-1)$-unform cover of $[n]$ we must have
\begin{equation}\label{eq:dual-BT-equivalent-2}\prod_{\sigma\in\overline{\sigma }^{\prime }}t_{\sigma }\ls (n!|K|)^{s-1}.\end{equation}
Combining \eqref{eq:dual-BT-equivalent-1} and \eqref{eq:dual-BT-equivalent-2} we see that $\prod_{i=1}^nt_{\{i\}}\gr n!|K|$.
On the other hand, $(\{i\}:i\in [n])$ is a $1$-uniform cover of $[n]$, and hence the reverse inequality is also true. Therefore,
\begin{equation}\label{eq:dual-BT-equivalent-3}\prod_{i=1}^nt_{\{i\}}= n!|K|.\end{equation}
Now, let $\tau\subseteq [n]$ and consider the $1$-uniform cover $\{ \tau\}\cup (\{i\}:i\notin\tau )$ of $[n]$. Using \eqref{eq:dual-BT-equivalent-3}
and the assumption that $t_{\tau }\gr \prod_{i\in\tau }t_{\{i\}}$ we have
$$n!|K|\gr t_{\tau }\cdot\prod_{i\notin \tau }t_{\{i\}}\gr \prod_{i\in\tau }t_{\{i\}}\cdot\prod_{i\notin\tau }t_{\{i\}}=
\prod_{i=1}^nt_{\{i\}}=n!|K|,$$
which implies that
\begin{equation}\label{eq:dual-BT-equivalent-4}t_{\tau }=\prod_{i\in\tau }t_{\{i\}}\end{equation}
for every $\tau\subseteq [n]$. The last set of equalities shows that if we consider the box $B=\prod_{i=1}^n[0,t_{\{i\}}]$ then we have
$|B|=\prod_{i=1}^nt_{\{i\}}=n!|K|$ and $|B\cap F_{\sigma }|=\prod_{i\in\sigma }t_{\{i\}}=t_{\sigma }\gr |\sigma |!|K\cap F_{\sigma }|$
for every $\sigma\subseteq [n]$. Then, if we set $\lambda_i=t_{\{i\}}/2$ and $C={\rm conv}(\{\pm \lambda_1e_1,\ldots ,\pm\lambda_ne_n\})$
we observe that $|C|=|K|$ and $|C\cap F_{\sigma }|\gr |K\cap F_{\sigma }|$ for every $\sigma\subseteq [n]$.
\end{proof}

\begin{application}[dual Ball's inequality]\label{appl:dual-BT-2}\rm Li and Huang proved in \cite{Li-Huang-2016} that for every
centered convex body $K$ in ${\mathbb R}^n$ and every even isotropic measure $\nu $ on $S^{n-1}$ one has
\begin{equation}\label{eq:li-huang}|K|^{n-1}\gr\frac{n!}{n^n}\exp\left (\int_{S^{n-1}}\log |K\cap u^{\perp }|\,d\nu (u)\right )\end{equation}
and they determined the equality cases. Their very interesting argument employs the continuous version of the Ball-Barthe inequality, due
to Lutwak, Yang and Zhang \cite{LYZ}, and a number of facts about the class of polar $L_p$-centroid bodies.
In the particular case where $u_1,\ldots ,u_m$ are unit vectors in ${\mathbb R}^n$ and
$c_1,\ldots ,c_m$ are positive real numbers that satisfy John's condition, one gets
\begin{equation}\label{eq:gen-lw-2}|K|^{n-1}\gr \frac{n!}{n^n}\prod_{i=1}^m|K\cap u_i^{\perp }|^{c_i}.\end{equation}
The latter inequality (in fact in a more general form) is a direct consequence of Theorem \ref{th:dual-2}. Given a convex
body $K$ in ${\mathbb R}^n$ with $0\in {\rm int}(K)$, we consider the subspaces $F_i=u_i^{\perp }$, and
since ${\rm dim}(F_i)=n-1$ and the $F_i$'s form an $(n-1)$-uniform cover of ${\mathbb R}^n$ with weights $c_1,\ldots ,c_m>0$,
using also the fact that $\sum_{i=1}^mc_i=n$ we immediately get
\begin{align}\label{eq:gen-lw-3}|K|^{n-1} &\gr\frac{1}{(n!)^s}\prod_{i=1}^m((n-1)!)^{c_i}\prod_{i=1}^m|K\cap u_i^{\perp }|^{c_i}
=\frac{[(n-1)!]^n}{(n!)^{n-1}}\prod_{i=1}^m|K\cap u_i^{\perp }|^{c_i}\\
\nonumber &=\frac{n!}{n^n}\prod_{i=1}^m|K\cap u_i^{\perp }|^{c_i}.
\end{align}
We can now use an approximation argument of Barthe from \cite{Barthe-2004} to deduce \eqref{eq:li-huang} from \eqref{eq:gen-lw-3}.
We sketch the idea of the proof and refer to Barthe's article for more details. Recall that a Borel measure $\nu $ on $S^{n-1}$
is called isotropic if $I_n=\int_{S^{n-1}}u\otimes u\,d\nu (u)$.
The fact that the vectors $u_j$ and the weights $c_j$ satisfy \eqref{eq:gen-lw-2} is equivalent to saying that the discrete
measure $\overline{\nu }$ with $\overline{\nu }(\{u_j\})=c_j$ is isotropic, i.e. $I_n=\int_{S^{n-1}}u\otimes u\,d\overline{\nu }(u)$.
Also, since
\begin{equation*}\int_{S^{n-1}}\log |K\cap u^{\perp }|\,d\overline{\nu } (u)=\sum_{i=1}^mc_i\log |K\cap u_i^{\perp }|
=\log\left (\prod_{i=1}^m|K\cap u_i^{\perp }|^{c_i}\right ),\end{equation*}
we may write \eqref{eq:gen-lw-3} in the equivalent form
\begin{equation}\label{eq:li-huang-discrete}|K|^{n-1}\gr\frac{n!}{n^n}\exp\left (\int_{S^{n-1}}\log |K\cap u^{\perp }|\,d\overline{\nu } (u)\right ).\end{equation}
In other words, \eqref{eq:li-huang} holds true for any discrete isotropic measure on $S^{n-1}$.

Now, let $\nu $ be an isotropic Borel measure on $S^{n-1}$. For any $\varepsilon >0$ we consider a maximal
$\varepsilon $-net $N_{\varepsilon }$ in $S^{n-1}$ and a partition $(C_u)_{u\in N_{\varepsilon }}$ of $S^{n-1}$
into Borel sets $C_u\subseteq B(u,\varepsilon )$, where $B(u,\varepsilon )$ is the geodesic ball with center $u$ and radius $\varepsilon $.
Then, we consider the measure
\begin{equation*}\nu_{\varepsilon }=\sum_{u\in N_{\varepsilon }}\nu (C_u)\delta_u,\end{equation*}
where $\delta_u$ is the Dirac mass at $u$.
Note that, for any continuous function $f:S^{n-1}\rightarrow {\mathbb R}$ we have that
\begin{equation*}\int_{S^{n-1}}f(u)\,d\nu_{\varepsilon } \longrightarrow \int_{S^{n-1}}f(u)\,d\nu\end{equation*}
as $\varepsilon\to 0$. In other words, $\nu_{\varepsilon }\rightarrow\nu $ weakly as $\varepsilon\to 0$.
If $T_{\varepsilon}=\int_{S^{n-1}}u\otimes u\,d\nu_{\varepsilon }(u)$ then for the measure
$\mu_{\varepsilon }=\sum_{u\in N_{\varepsilon }}\nu_{\varepsilon }(u)\|T_{\varepsilon }^{-1/2}(u)\|_2^2\delta_{v(u)}$
where $v(u):=T_{\varepsilon }^{-1/2}(u)/\|T_{\varepsilon }^{-1/2}(u)\|_2$ we have
\begin{equation*}I_n=\int_{S^{n-1}}T_{\varepsilon }^{-1/2}(u)\otimes T_{\varepsilon }^{-1/2}(u)\,d\nu_{\varepsilon }(u)
=\int_{S^{n-1}}v\otimes v\,d\mu_{\varepsilon }(v).\end{equation*}
Since $\|T_{\varepsilon }-I_n\|_{\ell_2^n\to\ell_2^n}\ls c_1(\varepsilon )$
for some constant $c_1(\varepsilon )$ that tends to $0$ as $\varepsilon\to 0$, we can check that for any continuous function $f:S^{n-1}\rightarrow {\mathbb R}$
\begin{equation*}\int_{S^{n-1}}f(u)\,d\mu_{\varepsilon } \longrightarrow \int_{S^{n-1}}f(u)\,d\nu\end{equation*}
as $\varepsilon\to 0$. Applying \eqref{eq:li-huang-discrete} for the discrete isotropic measure $\mu_{\varepsilon }$ we have
\begin{equation*}|K|^{n-1} \gr\frac{n!}{n^n}\exp\left (\int_{S^{n-1}}\log |K\cap u^{\perp }|\,d\mu_{\varepsilon } (u)\right )
\longrightarrow \frac{n!}{n^n}\exp\left (\int_{S^{n-1}}\log |K\cap u^{\perp }|\,d\nu (u)\right ).\end{equation*}
This proves \eqref{eq:li-huang}.
\end{application}

\bigskip

\noindent {\bf Acknowledgements.} The author would like to thank Apostolos Giannopoulos and Franck Barthe for useful discussions. He
also acknowledges support by the Department of Mathematics through a University of Athens Special Account Research Grant.

\bigskip

\bigskip

\footnotesize
\bibliographystyle{amsplain}

\begin{thebibliography}{100}
\footnotesize


\bibitem{Alonso-Artstein} \textrm{D.\ Alonso--Guti\'{e}rrez, S.\ Artstein--Avidan, B.\ Gonz\'{a}lez Merino, C.\ H.\ Jim\'{e}nez
and R.\ Villa}, {\sl Rogers--Shephard and local Loomis--Whitney type inequalities}, Preprint.
\bibitem{AGA-book} \textrm{S.\ Artstein--Avidan, A.\ Giannopoulos and V.\ D.\ Milman},
{\sl Asymptotic Geometric Analysis, Vol. I}, Mathematical Surveys and Monographs {\bf 202}, Amer. Math. Society (2015).

\bibitem{Ball-1991} {\rm K.\ M.\ Ball}, {\sl Shadows of convex bodies}, Trans. Amer. Math. Soc. {\bf 327} (1991), 891-901.

\bibitem{Ball-handbook} {\rm K. M.\ Ball}, {\sl Convex geometry and functional analysis},
Handbook of the geometry of Banach spaces (Johnson-Lindenstrauss eds), Vol. I, North-Holland, Amsterdam, (2001), 161-194.

\bibitem{Barthe-1998} {\rm F. Barthe}, {\sl On a reverse form
of the Brascamp-Lieb inequality}, Invent. Math. {\bf 134}
(1998), 335-361.
\bibitem{Barthe-2004} {\rm F. Barthe}, {\sl A continuous version of the Brascamp-Lieb inequalities},
Geometric aspects of functional analysis, 53-63, Lecture Notes in Math., {\bf 1850}, Springer, Berlin, 2004.



\bibitem{Bollobas-Thomason-1995} {\rm B.\ Bollob\'{a}s and A.\ Thomason}, {\sl Projections of bodies and hereditary properties of hypergraphs}, Bull.
London Math. Soc. {\bf 27} (1995), 417-424.
\bibitem{BGL} {\rm S.\ Brazitikos, A.\ Giannopoulos and D-M.\ Liakopoulos},
{\sl Uniform cover inequalities for the volume of coordinate sections and projections of convex bodies},
Advances in Geometry (to appear).








\bibitem{Li-Huang-2016} {\rm A-J.\ Li and Q.\ Huang}, {\sl The dual Loomis-Whitney inequality}, Bull. London Math. Soc. {\bf 48} (2016), 676-690.
\bibitem{LYZ} {\rm E.\ Lutwak, D.\ Yang and G.\ Zhang}, {\sl Volume inequalities for subspaces of $L_p$},
J. Differential Geom. {\bf 56} (2000), 111-132.
\bibitem {Loomis-Whitney-1949} {\rm L.\ H.\ Loomis and H.\ Whitney}, {\sl An inequality related to
the isoperimetric inequality}, Bull. Amer. Math. Soc. {\bf 55} (1949), 961-962.

\bibitem{Meyer-1988} {\rm M.\ Meyer}, {\sl A volume inequality concerning sections of convex sets}, Bull. London Math. Soc. {\bf 20}
(1988), 151-155.

\bibitem{Schneider-book} {\rm R.\ Schneider}, {\sl Convex Bodies: The Brunn-Minkowski Theory},
Second expanded edition. Encyclopedia of Mathematics and Its Applications 151, Cambridge University Press, Cambridge, 2014.



\end{thebibliography}

\bigskip

\medskip

\thanks{\noindent {\bf Keywords:}  Convex bodies, volume of projections and sections, Loomis-Whitney inequality, uniform cover inequality.}

\smallskip

\thanks{\noindent {\bf 2010 MSC:} Primary 52A20; Secondary 52A23, 52A40, 46B06.}

\bigskip

\bigskip

\noindent \textsc{Dimitris-Marios \ Liakopoulos}: Department of
Mathematics, National and Kapodistrian University of Athens, Panepistimiopolis 157-84,
Athens, Greece.

\smallskip

\noindent \textit{E-mail:} \texttt{dimliako1@gmail.com}

\end{document}